\documentclass[a4paper,final]{amsart}
\usepackage{amsfonts}
\usepackage{amssymb}
\usepackage[latin2]{inputenc}
\usepackage{enumerate}
\usepackage{amsmath}
\usepackage[notcite,notref]{showkeys}
\usepackage{amssymb}
\newtheorem{theorem}{Theorem}[section]
\newtheorem{lemma}[theorem]{Lemma}

\newtheorem{corollary}[theorem]{Corollary}

\newtheorem{conjecture}[theorem]{Conjecture}
\theoremstyle{definition}
\newtheorem{remark}[theorem]{Remark}

\newtheorem{definition}[theorem]{Definition}

\newcommand{\bt}{\begin{theorem}}
\newcommand{\et}{\end{theorem}}

\newcommand{\N}{\mathbb{N}}

\newcommand{\R}{\mathbb{R}}

\newcommand{\eps}{\varepsilon}

\def\beq{\begin{equation}}
\def\eeq{\end{equation}}

\newcommand{\bp}{\begin{proof}}
\newcommand{\ep}{\end{proof}}

\DeclareMathOperator{\diam}{diam}

\DeclareMathOperator{\cov}{cov}
\DeclareMathOperator{\trind}{ind}
\DeclareMathOperator{\hdim}{dim_{\mathsf{H}}}
\DeclareMathOperator{\tHD}{tHD}

\newcommand{\su}{\subset}

\newcommand{\iH}{\mathcal{H}}

\newcommand{\iM}{\mathcal{M}}

\newcommand{\cont}{\mathfrak{c}}
\renewcommand{\phi}{\varphi}

\title[Hausdorff dimension and Lipschitz maps]
{Hausdorff dimension of metric spaces and
Lipschitz maps onto cubes}

\author{Tam\'as Keleti}
\address{Institute of Mathematics\\
E\"otv\"os Lor\'and University\\
P\'azm\'any P\'eter s. 1/c, 1117
Budapest, Hungary}
\email{tamas.keleti@gmail.com}
\urladdr{http://www.cs.elte.hu/analysis/keleti}

\author{Andr\'as M\'ath\'e}
\address{Mathematics Institute, University of Warwick \\
Coventry, CV4~7AL, United Kingdom}
\email{A.Mathe@warwick.ac.uk}
\urladdr{http://homepages.warwick.ac.uk/~masibe/}

\author{Ond{\v r}ej Zindulka}
\address{Department of Mathematics, Faculty of Civil Engineering,
Czech Technical University, Th\'akurova 7, 160 00 Prague 6, Czech Republic}
\email{zindulka@mat.fsv.cvut.cz}
\urladdr{http://mat.fsv.cvut.cz/zindulka}

\begin{document}
\begin{abstract}
We prove that a compact metric space (or more generally an analytic subset
of a complete separable metric space) of Hausdorff dimension bigger than $k$
can be always mapped onto a $k$-dimensional cube by a Lipschitz map.
We also show that this does not hold for arbitrary separable metric spaces.

As an application we essentially answer a question of Urba\'nski
by showing that the transfinite Hausdorff dimension (introduced by
him) of an analytic subset $A$ of a complete separable metric space is
$\lfloor\hdim A\rfloor$ if $\hdim A$ is finite but not an integer,
$\hdim A$ or $\hdim A-1$ if $\hdim A$ is an integer and at least
$\omega_0$ if $\hdim A=\infty$.
\end{abstract}

\keywords{Lipschitz map, Hausdorff dimension, metric space,
monotone metric space, ultrametric space, analytic, ZFC,
transfinite Hausdorff dimension, H\"older function}

%\subsetbjclass[2010]{28A78, 54E40, 26A16}

\maketitle

\section{Introduction}

Which compact metric spaces $X$ can be mapped onto a $k$-dimensional cube
by a Lipschitz map? Since Lipschitz maps can increase the $k$-dimensional
Hausdorff measure by at most a constant multiple, a necessary condition is
$\iH^k(X)>0$, where $\iH^k$ denotes the $k$-dimensional Hausdorff measure.
Is this condition sufficient? In 1932
Kolmogorov \cite[last sentence of \S 6]{MR1512805} posed a conjecture
that would imply an affirmative
answer at least for $k=1$ and $X\subset\R^n$. However,
Vitushkin, Ivanov and Melnikov
\cite{MR0154965} (see also \cite{MR1313694} for a less concise proof) constructed
a compact subset of the plane with positive $1$-dimensional Hausdorff measure
that cannot be mapped onto a segment by a Lipschitz map.
Konyagin found in the 90's a simpler construction of an abstract compact
metric space with the same property but he has not published it.

By proving the following result we show that the condition $\hdim X>k$,
which is just a bit stronger than the necessary condition
$\iH^k(X)>0$, is already sufficient.

\begin{theorem}\label{compact}
If $(X,d)$ is a compact metric space with Hausdorff dimension larger than
a positive integer
$k$, then  $X$ can be mapped onto a
$k$-dimensional cube by a Lipschitz map.
\end{theorem}

In fact we shall prove a more general result (Theorem~\ref{analytic})
by allowing not only compact metric spaces but also Borel, or even analytic
subsets of complete separable metric spaces.
But some assumption about the metric space is needed: we show
(Theorem~\ref{t:sep}) that
there exist separable metric spaces with arbitrarily large Hausdorff dimension
that cannot be mapped onto a segment by a uniformly continuous function.
Surprisingly, our construction depends on a set theoretical
hypothesis that is independent of the standard ZFC axioms: if less than
continuum many sets of first category cannot cover the real line then we can
give an example in $\R^n$ (Theorem~\ref{t:constr1}), otherwise our example is a separable metric space
of cardinality less than continuum (Theorem~\ref{t:constr2}).

Recall that an \emph{ultrametric space} is a metric space in which the triangle
inequality is replaced with the stronger inequality
$d(x,y)\le \max(d(x,z),d(y,z))$. For compact ultrametric spaces we get
the following simple answer to the first question of the introduction.
\begin{theorem}\label{t:ultrametric}
A compact ultrametric space
can be mapped onto a $k$-dimensional
cube by a Lipschitz map if and only if
it has positive $k$-dimensional Hausdorff measure.
\end{theorem}
In fact, we prove (Corollary~\ref{c:mon_to_cube}) the above result for
the following more general (see Lemma~\ref{1-mon})
class of metric spaces.

\begin{definition}\label{d:mon}
A metric space $(X,d)$ is called \emph{monotone} if
there exists a linear order $<$ and a
constant $C$ such that
\begin{equation}\label{e:mon}
\diam([a,b])\le C \cdot d(a,b) \qquad (\forall a,b\in X),
\end{equation}
where $[a,b]$ denotes the closed interval
$\{x\in X: a\le x \le b\}$.

If this holds for a given $C$ then we can also say that the metric space
is \emph{$C$-monotone}.
\end{definition}

This notion has been introduced recently by the third author
in ~\cite{ZinCursed} and it was studied by him and
Nekvinda~\cite{MR2822419,NZOrder}.
Our results about monotone spaces are closely related to some of the
results of the mentioned papers but for the sake of
completeness we present here brief self-contained proofs.
Some further closely related
results will be published in \cite{ZinHolder} by the third author.

In section~\ref{sec:urb} we use our results to discuss the
\emph{Urba\'{n}ski conjecture}:
In his paper~\cite{MR2556034} Urba\'{n}ski defined the \emph{transfinite Hausdorff dimension}
$\tHD(X)$ of a separable metric space $X$ as the largest possible topological
dimension of a Lipschitz image of a subset of $X$ (cf.~\eqref{tHD}).
He proved that
if the Hausdorff dimension of $X$ is finite then
it is an upper bound for the
transfinite Hausdorff dimension of $X$ (cf. Theorem~\ref{t:urbanski})
and conjectured that, roughly speaking,
it is close to a lower bound (cf. Conjecture~\ref{c:urbanski}).
We show that the conjecture is correct for analytic subsets
of complete separable metric spaces, but
consistently fails in general.

\section{Nice large metric spaces can be mapped onto cubes}
\label{s:canbe}

First we prove the following result about mapping
monotone metric spaces onto an interval by a
H\"older function.

\begin{theorem}\label{t:Holder}
If $(X,d)$ is a compact
monotone metric space with positive %and finite
$s$-dimensional
Hausdorff measure (for some $s>0$)
then $X$ can be mapped onto a non-degenerate interval by an $s$-H\"older
function.
\end{theorem}

\begin{proof}
By Frostman lemma (see e.g. in \cite[Theorem 8.17]{MR1333890}) we can choose a nonzero finite Borel
measure $\mu$ on $X$ so that
$\mu(E)\le(\diam(E))^s$ for any $E\subset X$.
Since $X$ is a monotone metric space there exists a linear order $<$ and a
constant $C$ such that (\ref{e:mon}) of Definition~\ref{d:mon} holds.
It is easy to show (see also in \cite{NZOrder})
that any open interval $(a,b)=\{x\in X: a<x<b\}$ is open,
so any interval of $X$ is Borel.
For $x\in X$ let $g(x)=\mu((-\infty,x))$, where
$(-\infty,x)=\{y\in X : y<x \})$.
Then $g$ is $s$-H\"older since for any $a,b\in X$, $a<b$ we have
$$
0\le g(b)-g(a)=\mu([a,b))\le \diam([a,b))^s \le (C \cdot d(a,b))^s.
$$
Thus $g$ is continuous, so $g(X)\subset\R$ is compact.

Now we show that $g(X)$ is not a singleton.
Since $X$ is compact and the intervals of the form $(-\infty,b)$ and 
$(a,\infty)$ are open, 
there exists a minimal element $x_{-}$ 
and a maximal element $x_{+}$ in $(X,<)$. 
Then $g(x_{-})=\mu(\emptyset)=0$
and $g(x_{+})=\mu((-\infty,x_{+}))=\mu(X\setminus\{x_{+}\})$. 
Since $\mu(X)>0$ and $\mu(\{x_{+}\})\le (\diam (\{x_{+}\}))^s=0$
we get that indeed $g(x_{+})>g(x_{-})$.
 
Therefore all we need to prove is that there is no $u,v\in g(X)$ with
$u<v$ and $(u,v)\cap g(X)=\emptyset$.
Suppose there exist such  $u$ and $v$.
Since $X$ is a compact metric space, it is separable, let $S$ be a
countable dense
subset of $X$. Let $S_1 = \{s\in S: g(s)\le u\}$ and
$S_2=\{t\in S: g(t)\ge v\}$. Since  $(u,v)\cap g(X)=\emptyset$ we must
have $S=S_1\cup S_2$.
Since $S_1$ and $S_2$ are countable and $g(x)=\mu((-\infty,x))$ we get
that
$\mu\bigl(\bigcup_{s\in S_1}(-\infty,s)\bigr)\le u$ and
$\mu\bigl(\bigcap_{t\in S_2}(-\infty,t)\bigr)\ge v$.
Hence letting $D=\bigcap_{t\in S_2}(-\infty,t) \setminus \bigcup_{s\in S_1}(-\infty,s))$
we get $\mu(D)\ge v-u>0$.
On the other hand, the set $D$ must have at most two points since
if $x,y,z\in D$ and $x<y<z$ then $(x,z)$ would be a nonempty open set
not containing any point of the dense set $S=S_1\cup S_2$.
Since $\mu(E)\le(\diam(E))^s$ for any $E\subset X$, the singletons must have
zero $\mu$ measure, so we get
$\mu(D)=0$ contradicting the previously obtained $\mu(D)>0$.
\end{proof}

\begin{corollary}\label{c:mon_to_cube}
Let $X$ be a compact monotone metric space and let $k$ be a positive integer.
Then $X$ can be mapped onto the $k$-dimensional
cube $[0,1]^k$ by a Lipschitz map if and only if
$X$ has positive $k$-dimensional Hausdorff measure.
\end{corollary}

\begin{proof}
It is clear that $\iH^k(X)>0$ is a necessary condition.

To prove that it is sufficient note that
by the previous theorem there exists a $k$-H\"older map $g:X\to\R$ such that
$g(X)=[0,1]$. It is well-known 
(see e.g. \cite[Theorem 4.55]{milne})
that there exists a $\frac{1}{k}$-H\"older Peano
curve $h:[0,1]\to[0,1]^k$. 
Then the composition $h\circ g$ is a Lipschitz map that maps $X$ onto
$[0,1]^k$.
\end{proof}

Nekvinda and Zindulka \cite{NZOrder} proved that an
ultrametric space is always a monotone metric space.
For compact ultrametric spaces even the following stronger result
can be proved fairly easily.

\begin{lemma}\label{1-mon}
Any compact ultrametric space $(X,d)$ is $1$-monotone.
\end{lemma}

\begin{proof}
The following tree structure is well-known but as it can be obtained
quickly we give a self-contained proof.

Let $D=\diam X$. If $D=0$ then let $k=1$ and $X_1=X$.
Otherwise, since $d$ is an ultrametric,
the relation $d(x,y)<D$ is an equivalence relation. The equivalence classes
are open, $X$ is compact, so
there are only finitely many equivalence classes: $X_1,\ldots,X_k$.
Thus $X_1,\ldots,X_k$ are closed and compact as well.
Since $X$ is compact, there exist two points with distance $D$, so $k\ge 2$,
unless $D=0$.
Note also that if $a$ and $b$ are from distinct equivalence classes then
$d(a,b)=D$.

Since each equivalence class is a compact ultrametric space we can do the same
for each of them. This way we get a tree of clopen sets $X_{i_1\ldots i_m}$
with the property that
for a fixed $m$ these sets give a partition of $X$ and
if $a\in X_{i_1\ldots i_m j}$, $b\in X_{i_1\ldots i_m j'}$
and $j\neq j'$ then $d(a,b)=\diam(X_{i_1\ldots i_m})$.
The first property implies that for any $x\in X$ there exists a unique
sequence $J(x)=(i_1,i_2\ldots$), so that $x\in X_{i_1\ldots i_m}$ for every $m$.
On the other hand, $\lim_{m\to\infty}\diam(X_{i_1\ldots i_m})=0$,
since otherwise picking one point from each
$X_{i_1\ldots i_{m}}\setminus X_{i_1\ldots i_{m+1}}$ we would get an infinite discrete
subspace. 
Hence $\{x\}=\bigcap_{m=1}^\infty X_{i_1\ldots i_m}$,
therefore the function $J$ is injective. 

This injectivity of $J$ ensures that the following pull-back
of the lexicographic order via $J$ is an order on $X$:
let $x<y$ if $J(x)$ is smaller than
$J(y)$ in the lexicographical order. 
We need to show
that for any $a,b\in X$ we have
$\diam([a,b])=d(a,b)$. Let $(i_1,\ldots,i_m)$ be the longest common
initial segment of $a$ and $b$. Then, as we saw above,
$d(a,b)=\diam(X_{i_1\ldots i_m})$. On the other hand
$[a,b]\subset X_{i_1\ldots i_m}$, so we get
$
d(a,b)\le\diam([a,b])\le\diam(X_{i_1\ldots i_m})=d(a,b),
$
which completes the proof.
\end{proof}

The following result is a weaker version of \cite[Theorem 1.4]{MendelNaor}.

\begin{theorem}[Mendel and Naor \cite{MendelNaor}]\label{t:MN}
For every compact metric space $(X,d)$ and $\eps>0$ there exists a closed
subset $Y\subset X$ such that $\hdim Y\ge (1-\eps)\hdim X$ and $(Y,d)$
is bi-Lipschitz equivalent to an ultrametric space.
\end{theorem}

We will also need the following result.

\begin{theorem}[Howroyd \cite{MR1317515}]
Let $s>0$. If an analytic subset of 
a separable complete metric space 
is of infinite $s$-dimensional Hausdorff meausure
then it has a compact subset of finite and positive
$s$-dimensional Hausdorff measure.
\end{theorem}

Now we are ready to prove the main result of the paper.

\begin{theorem}\label{analytic}
Let $A$ be an analytic subset of a separable complete metric space $(X,d)$,
and let $k$ be a positive integer.
If $\hdim A>k$ then $A$ can be mapped onto the
$k$-dimensional cube $[0,1]^k$ by
a Lipschitz map.
\end{theorem}

\begin{proof}
Let $s\in(k,\hdim A)$.
By the above Howroyd theorem, $A$ has a compact subset $C$ with finite
and positive $s$-dimensional Hausdorff measure.
By the above Mendel--Naor theorem, $C$ has a closed subset $E$ with
$\hdim E>k$ that is bi-Lipschitz equivalent to an ultrametric space.
Applying Howroyd's theorem again we get a compact subset $B$ of $E$ with
positive and finite $k$-dimensional Hausdorff measure.
Clearly, $B$ is also bi-Lipschitz equivalent to a compact ultrametric space $Y$.
By Lemma~\ref{1-mon} and Corollary~\ref{c:mon_to_cube},
$Y$ can be mapped onto $[0,1]^k$ by a Lipschitz map.
Since $B\subset A$ is bi-Lipschitz equivalent to $Y$,
this means that $B$ can be also mapped onto $[0,1]^k$ by a
Lipschitz map. Since real valued Lipschitz functions can be always extended
as Lipschitz functions, by extending the coordinate functions we get a
Lipschitz function that maps $A$ onto $[0,1]^k$.
\end{proof}

\begin{remark}
As we saw in the introduction,
in the above theorem the condition $\hdim A>k$ cannot be replaced by
the condition that $A$ has positive $k$-dimensional Hausdorff measure,
not even if $A$ is a compact subset of $\R^n$.
% THIS CAN BE UNCOMMENTED IF WE DON'T WANT A 
%'FURTHER DIRECTIONS FOR RESEARCH'SECTION
%However, for $k=n$ this stronger statement might be also true. It is a
%well-known long standing conjecture of Laczkovich \cite{MR1147388} that a
%measurable set
%with positive Lebesgue measure in $\R^n$ can be always mapped onto
%an $n$-dimensional cube by a Lipschitz map.
%This is known to be true for $n\le 2$
%(see \cite{MR2185733} and \cite{MR1425223}),
%and there is recent (still unpublished) progress for $n\ge 3$ due to
%Cs\"ornyei and Jones.
\end{remark}

\begin{remark}
The following argument shows that
if $(X,d)$ is the Euclidean space $\R^n$, or more generally if
it is a complete doubling metric space then we can also prove
Theorem~\ref{analytic}
without using the recent deep theorem of Mendel and Naor
(instead we use more classical theorems of Assouad and Mattila,
and we still need Howroyd's theorem).

First, let $A$ be an analytic set in $\R^n$ with $\hdim A>k$.
Let $C\subset A$ be compact with $\hdim C>k$.
Choose $S$ to be a self-similar set in $\R^n$ with the strong separation
condition
(which means that $S$ is the disjoint union of sets similar to $S$) with
$\hdim S>\max(n-(\hdim C-k), (n+1)/2)$.
By a theorem of Mattila \cite[Theorem 13.11]{MR1333890}
there exists
an isometry $\varphi$ such that $\hdim C\cap\varphi(S) > k$.
It is easy to check that a self-similar set with the strong separation
condition is bi-Lipschitz equivalent to an ultrametric space,
so $C\cap\varphi(S)$ is bi-Lipschitz equivalent
to a compact ultrametric space $X$.
Since ultrametric spaces are monotone this implies
by Corollary~\ref{c:mon_to_cube} that $C\cap\varphi(S)$ can be mapped onto
$[0,1]^k$ by a Lipschitz function.
By extending the Lipschitz function onto $A$ we get a
Lipschitz function that maps $A$ onto $[0,1]^k$.

By Theorem~\ref{t:Holder} we also get that $A$ can be mapped onto $[0,1]$
by an $s$-H\"older function for some $s>k$.
By the Assouad embedding theorem \cite{MR763553} for any doubling metric space
$(X,d)$ and $\eps>0$ the metric space $(X,d^{1-\eps})$ admits a bi-Lipschitz
embedding into a Euclidean space. Combining these results we get that if
$B$ is an analytic subset of a complete doubling metric space $(X,d)$
and $\hdim(B)>k$ then $B$ can be mapped
onto $[0,1]$ by a $k$-H\"older function. Then, composing this map with a
$\frac1k$-H\"older Peano curve as in the proof of Corollary~\ref{c:mon_to_cube},
we get a Lipschitz map from $B$ onto $[0,1]^k$.
\end{remark}

\section{Large metric spaces that cannot be mapped onto a segment}
\label{cannot}

The main result (Theorem~\ref{analytic}) of the previous section said that 
reasonably nice metric spaces of large Hausdorff dimension can be
always mapped onto large dimensional cube by a Lipschitz map. 
The following result, which is the main result of this section, shows that some
assumption on the metric space is necessary, even if we allow not only
Lipschitz functions but also uniformly continuous functions.

\begin{theorem}\label{t:sep}
There exist separable metric spaces with arbitrarily large Hausdorff dimension
that cannot be mapped onto a segment by a uniformly continuous function.
\end{theorem}

We obtain the above theorem by proving the following two results.

\begin{theorem}\label{t:constr1}
If less than continuum many sets of first category cannot cover $\R$,
then for any $n$ there exists a set $A\su\R^n$ of Hausdorff dimension $n$
that cannot be mapped onto a segment by a uniformly continuous function.
\end{theorem}

\begin{theorem}\label{t:constr2}
If less than continuum many sets of first category can cover $\R$,
then there exist separable metric spaces of arbitrarily high Hausdorff dimension 
such that their cardinality is less than continuum, consequently they
cannot be mapped onto a segment by any funtion.
\end{theorem}

The hypotheses of the above two theorems are abbreviated by $\cov\iM=\cont$ and
$\cov\iM<\cont$, where $\cont$ denotes the cardinality continuum, 
$\iM$ is the collection of subsets of $\R$ of
first category, and $\cov\iM$ stands for the least cardinality of a collection
of sets of first category that can cover $\R$. It is well known (see e.g. in 
\cite[Chapter 7]{MR1350295}) that both $\cov\iM=\cont$ and
$\cov\iM<\cont$ are consistent with the standard ZFC axioms of set theory. 
Note that the 
Continuum Hypothesis (CH) clearly implies $\cov\iM=\cont$ but it is also
consistent with ZFC that CH fails but $\cov\iM=\cont$ holds (see e.g. in 
\cite[Chapter 7]{MR1350295}). 

First we prove Theorem~\ref{t:constr1}.
The following result is probably known but for completeness we present a proof.

\begin{theorem}\label{ch}
If Continuum Hypothesis holds then for any $n$ there exists a set $A\subset\R^n$
such that for any continuous function $f:\R^n\to\R$ the set
$f(A)$ does not contain any interval
and $A$ is not a Lebesgue null set.

In fact, instead of the \emph{CH}, it is enough to assume the following hypothesis,
which clearly follows from \emph{CH}.
\begin{equation}\label{*}
\begin{minipage}{0.85\textwidth}
  Less than continuum many closed measure zero sets
  and a set of measure zero cannot cover $\R^n$.
\end{minipage}\tag{$\star$}
\end{equation}
\end{theorem}
\begin{proof}
Let $\{f_{\alpha}:\alpha<\cont\}$ and $\{N_{\alpha}:\alpha<\cont\}$ be enumerations of the
collection of $\R^n\to\R$ continuous functions and the collection of
Lebesgue null Borel subsets of $\R^n$, respectively.

By transfinite induction for every $\alpha<\cont$
we construct points $x_{\alpha}\in\R^n$ and
$y_{\alpha}\in (0,1)$ such that
\begin{enumerate}
\item[(i)] $x_{\alpha}\not\in N_{\alpha}$,
\item[(ii)] $x_{\alpha}\not\in\bigcup_{\beta<\alpha}f_{\beta}^{-1}(\{y_{\beta}\})$,
\item[(iii)] $y_{\alpha}\not\in f_{\alpha}(\{x_{\beta} : \beta\le \alpha\})$ and
\item[(iv)] $f_{\alpha}^{-1}(\{y_{\alpha}\})$ has Lebesgue measure zero.
\end{enumerate}
In the $\alpha$-th step we suppose that (i)--(iv) hold for smaller indices.
Since for each $\beta<\alpha$ the set $f_{\beta}^{-1}(\{y_{\beta}\})$
is a closed set of measure zero, \eqref{*} implies that
we can choose $x_{\alpha}$ so that (i) and (ii) hold.
We can choose $y_{\alpha}\in(0,1)$ so that (iii) and (iv) hold
since more than countably many of the pairwise disjoint closed sets
$f_{\alpha}^{-1}(\{t\})$ $(t\in(0,1))$ cannot have positive measure.

Let $A=\{x_{\alpha}:\alpha<\cont\}$.
This set cannot have zero measure by (i).
If there exists a continuous function $f:\R^n\to\R$ such that
$f(A)$ contains an interval then the image of a linear transformation contains
$(0,1)$, so there exists an $\alpha<\cont$ for which $f_{\alpha}(A)\supset (0,1)$.
But this is impossible since $f_{\alpha}(x_{\beta})\neq y_{\alpha}$ for any
$\beta\le\alpha$ by (iii) and for any $\beta>\alpha$ by (ii), so
$y_{\alpha}\not\in f_{\alpha}(A)$.
\end{proof}

Since real valued uniformly continuous functions can be always extended
we get the following.

\begin{corollary}\label{cor:ch}
If Continuum Hypothesis (or~\eqref{*} of Theorem~\ref{ch})
holds then for any $n$ there exists a set $A\subset\R^n$
such that no subset of $A$ can be mapped onto a segment
by a uniformly continuous map and $A$ is not a Lebesgue null set,
so it has Hausdorff dimension $n$.
\end{corollary}

Corollary~\ref{cor:ch} together with the following result clearly completes
the proof of Theorem~\ref{t:constr1}. The result itself is known but we
have not found it in the literature in this exact form, so for completeness 
we show how it follows from published theorems.

\begin{theorem}
The hypothesis $\cov\iM=\cont$ implies \eqref{*}.
\end{theorem}

\begin{proof}
Let $2^\omega$ be the product space $\{0,1\}\times\{0,1\}\times\ldots$
and let us consider the natural 
uniformly distributed product probability measure on it.

By a well-known theorem of Bartoszy\'nski and Shelah 
\cite{BS} (see also  \cite[Theorem 2.6.14]{MR1350295})
$\cov\iM$ is also the least cardinality of a collection
of closed sets of measure zero in $2^{\omega}$ such that the union is not a 
set of measure zero. 
(In this theorem $\iM$ is meant on $2^\omega$ but it does not change $\cov \iM$
since $(\R,\iM_R)$ and $(2^{\omega},\iM_{2^\omega})$ are isomorphic, see e.g. 
\cite[522V(b)(i)]{Fremlin5}.)
By this theorem $\cov\iM=\cont$ implies that 
\begin{equation}\label{**}
\begin{minipage}{0.9\textwidth}
  in $2^{\omega}$ less 
  than continuum many closed measure zero sets
  and a set of measure zero cannot cover $2^{\omega}$.
\end{minipage}\tag{$\star\star$}
\end{equation}

We need to prove that this implies that the same is true
in $\R^n$. 
For this we need a continuous measure preserving
map $f:2^{\omega}\to[0,1]^n$. Suppose that we have such a function
and \eqref{*} is false, so there exists a decomposition $\R=N\cup \bigcup_{\alpha\in I} F_{\alpha}$, where $N$ is a set of measure zero, $I$ has
cardinality less than continuum, and each $F_{\alpha}$ is a 
closed set of measure zero. Then we get decomposition
$
2^\omega=f^{-1}(\R^n)=
f^{-1}(N)\cup \bigcup_{\alpha\in I} f^{-1}(F_{\alpha}),
$
which contradicts \eqref{**}.

So it remains to show a continuous measure preserving
map $f:2^{\omega}\to[0,1]^n$. This is well known and easy:
let the $j$-th coordinate of $f((a_k))$ be 
$\sum_{i=0}^\infty a_{in+j}2^{-i-1}$.
\end{proof}

\begin{remark}\label{r:Corazza}
Some set theoretical assumption is needed in
Theorem~\ref{t:constr1}, it
cannot be proved in ZFC:
Corazza \cite{MR982239} constructed a model in which every
%IN ONDREJ'S SUGGESTION IT IS NOT CLEAR WHAT HAS POS. H-DIM.
%subset of a Euclidean space of positive Hausdorff dimension 
positive Hausdorff dimensional subset of a Euclidean space 
can be mapped onto $[0,1]$ by a uniformly continuous function.
%
%THIS CAN BE UNCOMMENTED IF WE DON'T WANT A 
%'FURTHER DIRECTIONS FOR RESEARCH'SECTION
%We do not know whether the existence of finite Hausdorff %dimensional
%examples can be proved in ZFC.
%We do not know either whether Theorem~\ref{t:constr1} can be %proved in ZFC
%for Lipschitz maps.
\end{remark}

Now it remains to prove Theorem~\ref{t:constr2}.

Let $\iH^{\varphi}$ denote the
Hausdorff measure with strictly increasing bijective gauge function
$\varphi:[0,\infty)\to[0,\infty)$.
A metric space $X$ is of \emph{strong measure zero} if $\iH^{\varphi}(X)=0$
for any gauge function $\varphi$.

We construct our example in $\omega^\omega$.
For $x\neq y$ in $\omega^\omega$ denote by $|x\wedge y|$ the length of
the common initial segment of $x$ and $y$. Note that, for any decreasing
function $g:\{1,2,\ldots\}\to(0,\infty)$ with $\lim_n g(n)=0$, by letting
$d_g(x,y)=g(|x\wedge y|+1)$ we get a separable metric space
$(\omega^\omega,d_g)$. One of the classical metrics on $\omega^\omega$ is
$d_{g_0}$ for $g_0(m)=1/m$.

The following theorem clearly implies 
Theorem~\ref{t:constr2}, so this will also complete
the proof of Theorem~\ref{t:sep}.

\begin{theorem}\label{t:cov}
For any gauge function $\varphi$
there exists a separable metric space $(X,d)$ of cardinality $\cov\iM$
with $\iH^{\varphi}((X,d))>0$.
\end{theorem}
\begin{proof}
Fremlin and Miller \cite[Theorem 5]{MR954892} proved that $\cov\iM$ is also
the least cardinality of a subset of $(\omega^\omega,d_{g_0})$ that is not a
strong measure zero subspace. Therefore there exists $H\subset \omega^\omega$
of cardinality $\cov\iM$ such that $(H,d_{g_0})$ is not a strong measure
zero metric space, so
there exists a gauge function $\varphi_0$ such that $\iH^{\varphi_0}(H,d_{g_0})>0$.
Set $g=\varphi^{-1}\circ \varphi_0 \circ g_0$. Then $\varphi_0\circ g_0=\varphi \circ g$
and so  $\iH^{\varphi}(H,d_{g}) = \iH^{\varphi_0}(H,d_{g_0})>0$.
Therefore $(H,d_{g})$ is a
separable metric space of cardinality $\cov\iM$
with $\iH^{\varphi}((H,d_g))>0$.
\end{proof}

\section{Transfinite Hausdorff dimension}\label{sec:urb}

Urba\'nski \cite{MR2556034} introduced  the transfinite Hausdorff dimension ($\tHD$)
of a metric space $X$ in the following way:
\begin{equation}\label{tHD}
\tHD(X) = \sup\{\trind f(Y) :\text{$Y\subset X$, $f:Y\to Z$ Lipschitz, $Z$ a metric space}\},
\end{equation}
where $\trind$ denotes the transfinite small inductive topological dimension
(see e.g. in \cite{MR1363947}). He showed the following connection between
Hausdorff dimension and transfinite Hausdorff dimension.

\begin{theorem}[Urba\'nski {\cite[Theorem 2.8]{MR2556034}}]\label{t:urbanski}
If $X$ is a metric space with finite Hausdorff dimension then
$\tHD(X)\le\lfloor\hdim X\rfloor$, where
$\lfloor.\rfloor$ denotes the floor function.
\end{theorem}

After noticing
that equality is not always true (if $C$ is a Cantor set
with zero Lebesgue measure but Hausdorff dimension $1$
then $\tHD(C)=0)$
Urba\'nski stated the
following conjecture.

\begin{conjecture}[Urba\'nski {\cite[Conjecture 6.1]{MR2556034}}]\label{c:urbanski}
If $X$ is a metric space with finite Hausdorff dimension then
$\tHD(X)\ge \lfloor\hdim X\rfloor-1$.
Consequently
$\tHD(X)\in\{\lfloor \hdim X \rfloor - 1, \lfloor \hdim X \rfloor\}$.
\end{conjecture}

It is not hard to see that one cannot prove this conjecture without some
additional assumptions:
It is well-known (see e.g.~in \cite{MR1350295}) that
for any $n$ the existence of a set $X\subset\R^n$ with positive Lebesgue outer
measure and cardinality less than continuum is consistent with ZFC.
Then $\hdim X=n$ but $\tHD(X)=0$ since any image of any subset of $X$ has
cardinality less than continuum
and it is easy to check that any set of cardinality
less than continuum must have zero topological dimension
(since we can easily find balls of arbitrarily small radius with empty boundary
around any point).

Theorem~\ref{analytic} immediately implies that if we have some reasonable
assumption about the metric space then even the following stronger form
of Conjecture~\ref{c:urbanski} holds.

\begin{theorem}\label{conjholds}
If $A$ is an analytic subset of a complete separable metric space then
$\tHD(A)\ge\lceil\hdim A\rceil-1$,
where $\lceil . \rceil$ denotes the ceiling function.
\end{theorem}

Combining Theorems~\ref{t:urbanski} and \ref{conjholds} we get the
following.
\begin{corollary}\label{final}
Let $A$ be an analytic subset of a complete separable metric space.
\begin{itemize}
\item
If $\hdim A$ is finite but not an integer then
$\tHD(A)=\lfloor\hdim A\rfloor$,
\item
if $\hdim A$ is an integer then $\tHD(A)$ is $\hdim A$ or $\hdim A-1$,
and
\item
if  $\hdim A=\infty$ then $\tHD(A)\ge\omega_0$.
%\qedthm
\end{itemize}
\end{corollary}

\begin{remark}
Corollary~\ref{final} is sharp in the sense that if $\hdim A=n\in\N$ then
$\tHD(A)$ can be either $n$ or $n-1$: the trivial $A=\R^n$ is clearly
an example for the former one and we claim that any $A$ with
zero $n$-dimensional Hausdorff measure is an example for the latter one.
It is well-known (see e.g. \cite{MR0006493}) that the topological dimension of
a metric space with zero $n$-dimensional Hausdorff measure is at most $n-1$.
Therefore
if $A$ has zero $n$-dimensional Hausdorff measure then the Lipschitz image of
any of its subsets also has zero $n$-dimensional Hausdorff measure, so
its topological dimension is at most $n-1$, which means that $\tHD(A)\le n-1$.
\end{remark}

Our final goal is to show (in ZFC) that Theorem~\ref{conjholds} and Corollary~\ref{final}
do not hold for a general separable metric space.
The proof of the following useful observation uses a well-known argument that
can be found for example in \cite{BBE} or~\cite[4.2]{MR2009767}.

\begin{lemma}\label{l:null}
For any metric space $X$,
$\tHD(X)=0$ if and only if $X$ cannot be mapped onto a
segment by a Lipschitz map.
\end{lemma}

\begin{proof}
It is clear that if $\tHD(X) = 0$ then X cannot be mapped onto a segment by
a Lipschitz map. To prove the converse suppose that $Y\subset X$ and $f:Y\to Z$
is a Lipschitz map
onto a metric space $(Z,\rho)$ of positive topological dimension.
The latter implies that there is $z_0\in Z$ and $r>0$ such that every sphere
$\{z\in Z:\rho(z_0,z)=s\}$ of radius $s\leq r$ is nonempty.
It follows that the Lipschitz function $g(z)=\min(\rho(z_0,z),r)$
maps $Z$ onto $[0,r]$.
Thus a Lipschitz extension of $g\circ f$ maps $X$ onto $[0,r]$.
\end{proof}

Combining Theorem~\ref{t:sep} and Lemma~\ref{l:null} we get the following.

\begin{theorem}
\label{t:last}
There exist separable metric spaces with zero transfinite Hausdorff dimension
and arbitrarily large Hausdorff dimension.
\end{theorem}

%THIS CAN BE UNCOMMENTED IF WE DON'T WANT A 
%'FURTHER DIRECTIONS FOR RESEARCH'SECTION
%\begin{remark}
%We do not know whether in ZFC one can give an example with finite Hausdorff
%dimension and we cannot disprove the original
%form of Conjecture~\ref{c:urbanski} in ZFC either.
%\end{remark}

\section{Further directions for research}

There are (at least) two possible natural directions of 
future research:

1. We still do not know the answer to the first question of
the introduction: 
Which compact metric spaces $X$ can be mapped onto a $k$-dimensional cube by a Lipschitz map? 

Our Theorem~\ref{compact} gives that $\hdim X>k$ is a 
sufficient condition but this is clearly not a necessary condition (the $k$-dimensional cube itself is a trivial counter-example). 

As we saw in the introduction, the trivial necessary condition that the $k$-dimensional Hausdorff measure $\iH^k(X)$ is positive is not sufficient, 
not even if $X$ is a compact subset of $\R^n$.
However, for $k=n$ this might be sufficient: 
it is a
well-known long standing conjecture of Laczkovich \cite{MR1147388} that a
measurable set
with positive Lebesgue measure in $\R^n$ can be always mapped onto
an $n$-dimensional cube by a Lipschitz map.
This is known to be true for $n\le 2$
(see \cite{MR2185733} and \cite{MR1425223}),
and there is recent (still unpublished) progress for $n\ge 3$ due to
Cs\"ornyei and Jones.

Our Corollary~\ref{c:mon_to_cube} gives that for compact
monotone metric spaces $X$ 
(in particular for compact ultrametric spaces)
the condition $\iH^k(X)>0$ is necessary
and sufficient. 
It would be nice to know more about 
those compact metric spaces for which the 
condition $\iH^k(X)>0$ is necessary
and sufficient. A full characterization of these spaces
seems to be hopeless since, as we saw above, we do not
even know if the compact subspaces of $\R^k$ have this property. 
The other difficulty is that it is hard to construct
compact metric spaces that do not have this property,
only the two ingenious constructions mentioned in
the introduction are known:
the construction of Vitushkin, Ivanov and Melnikov \cite{MR0154965} (see also \cite{MR1313694}) is really
complicated, the unpublished construction of Konyagin 
is not simple either.
It would be nice to have simpler examples.

2. Although the main results Theorem~\ref{t:sep},
Corollary~\ref{final} and Theorem~\ref{t:last} 
of Sections \ref{cannot} and \ref{sec:urb}  do not
depend on any set theoretical hypothesis (in other words
these results are proved in ZFC), there are several 
statements that we cannot prove in ZFC, only under some
hypothesis. 
As it is explained in Remark~\ref{r:Corazza}, Theorem~\ref{t:constr1}
cannot be proved in ZFC.
But for the following weaker and weaker statements
we do not know
if they can be proved in ZFC or they are independent of ZFC
(which is the case if the negation of the statement is
also consistent with ZFC). 

\begin{enumerate}

\item[(i)]\label{s1}
For any $n$ there exists a set $A\subset\R^n$
with Hausdorff dimension $n$ that cannot be mapped
onto a segment by a Lipschitz function.

\item[(ii)]\label{s2}
There exist separable metric spaces with arbitrarily
large finite Hausdorff dimension that cannot be mapped
onto a segment by a Lipschitz function.

\item[(ii')]\label{s3}
There exist separable metric spaces with zero transfinite Hausdorff dimension and arbitrarily
large finite Hausdorff dimension. 

\item[(iii)]\label{s4}
The original form (Conjecture~\ref{c:urbanski}) of Urbanski's conjecture is false.

\end{enumerate}

Note that (ii) $\Leftrightarrow$ (ii') holds by Lemma~\ref{l:null},
the implications (i) $\Rightarrow$ (ii) and
(ii') $\Rightarrow$ (iii) are clear, 
and by Theorem~\ref{t:constr1} $\cov\iM=\cont$ 
(so in particular CH as well) implies 
all four statements.

%As we saw in Theorem~\ref{t:constr1} and Remark~\ref{r:Corazza}
%the stronger statement in which "Lipchitz" is replaced
%by "uniformly continuous" is independent of ZFC.
%
%
%
%As we explain in Remark~\ref{r:Corazza}, Theorem~\ref{t:constr1}
%cannot be proved in ZFC. But we do not know this for Lipschitz functions, instead
%of uniformly continuous function, 

\section*{Funding}
This work was supported by the Hungarian Scientific Research Fund [72655 to T.K. and A.M.]; the Engineering and Physical Sciences Research Council [EP/G050678/1 to A.M.]; 
and the Department of Education of the Czech Republic [research project BA MSM 210000010 to O.Z.].

\section*{Acknowledgement}
We are indebted to M\'arton Elekes for some illuminating discussions and to the anonymous referees for valuable comments.

%%%%%%%%%%%%%%%%%%%%%%%%%%%%%%%%%%%%%%%%%%%%%%%%%%%%%%%%%%%
%%%%%%%%%%%%%%%%%%%%%%%%%%%%%%%%%%%%%%%%%%%%%%%%%%%%%%%%%%%
% BIBLIOGRAPHY -------------------------------------------
%\bibliographystyle{plain}
%\bibliography{kmz}
%\end{document}

\providecommand{\bysame}{\leavevmode\hbox to3em{\hrulefill}\thinspace}
\providecommand{\MR}{\relax\ifhmode\unskip\space\fi MR }
%% \MRhref is called by the amsart/book/proc definition of \MR.
%\providecommand{\MRhref}[2]{%
%  \href{http://www.ams.org/mathscinet-getitem?mr=#1}{#2}
%}
%\providecommand{\href}[2]{#2}

\end{document}